\documentclass[11pt]{amsart}
\usepackage{amsmath, amsthm, amsfonts, amssymb, amscd}
\usepackage{marginnote, todonotes}
\usepackage[colorlinks=true, citecolor=blue, linkcolor=red,pagebackref, hyperindex]{hyperref}
\usepackage{fullpage}

\usepackage{amsxtra}     
\usepackage{epsfig}
\usepackage{verbatim}
\usepackage[all,cmtip]{xy}
\usepackage{mathrsfs}
\usepackage{enumerate}

\newcommand{\ZZ}{{\mathbb Z}}

\newcommand{\m}{{\mathfrak m}}

\newcommand{\Tor}[4]{\operatorname{Tor}_{#1}^{#2}(#3,#4)}

\newcommand{\depth}{\operatorname{depth}}

\theoremstyle{plain}
\newtheorem{theorem}{Theorem}[section]
\newtheorem{corollary}[theorem]{Corollary}
\newtheorem{lemma}[theorem]{Lemma}

\newtheorem{proposition}[theorem]{Proposition}

\theoremstyle{definition}

\newtheorem{remark}[theorem]{Remark}
\newtheorem{definition}[theorem]{Definition}

\newtheorem{example}[theorem]{Example}
\newtheorem*{acknowledgement}{Acknowledgments}

\numberwithin{equation}{theorem}

\theoremstyle{remark}

\newtheorem{question}[theorem]{Question}

\renewcommand{\geq}{\geqslant}

\renewcommand{\leq}{\leqslant}

\begin{document}

\title{On monomial Golod ideals}

\author[H. Dao]{Hailong Dao}
\address{Department of Mathematics, The University of Kansas, Lawrence, KS 66045,  U.S.A.}
\email{hdao@ku.edu}
\urladdr{}

\author[A. De Stefani]{Alessandro De Stefani}
\address{Department of Mathematics, University of Nebraska, 203 Avery Hall, Lincoln, NE 68588}
\email{adestefani2@unl.edu}

\date{\today} 


\begin{abstract}
We study ideal-theoretic conditions for a monomial ideal to be Golod. For ideals in a polynomial ring in three variables, our criteria give a complete characterization. Over such rings, we show that the product of two monomial ideals is Golod. 
\end{abstract}

\maketitle

\thispagestyle{empty}


\section{Introduction}
\label{sec:Introduction}
Let $k$ be a field, and $Q=k[x_1,\dots, x_n]$ be a polynomial ring on $n$ variables over $k$, with $\deg(x_i)=1$ for all $i$. We denote by $\m = (x_1,\dots,x_n)$ the homogeneous maximal ideal of $Q$. Let $I \subseteq \m^2$ be a homogeneous ideal and $R=Q/I$. Serre proved a coefficient-wise inequality of formal power series for the Poincare series of $R$:
$$P^R_k(t):= \sum_{i\geq 0} \dim_k \Tor iRkk t^i\ll  \frac{(1+t)^n}{1-\sum_{i>0}\dim_k \Tor iQkR t^{i+1}} $$

When equality happens, the ring $R$ (and the ideal $I$) are called Golod. The notion is defined and studied extensively in the local setting, but in this paper we shall restrict ourselves to the graded situation.  Golod rings and ideals have attracted increasing attention recently, see \cite{CV, HH, D, Fr, Ka}, but they remain mysterious even when $n=3$. For instance, we do not know if the product of any two homogeneous ideals in $Q=k[x,y,z]$ is Golod. Another reason for the increasing interest is their connection to moment-angle complexes, for example see \cite{DS,IK,GPTW}.

It was asked by Welker whether it is always the case that the product of two proper homogenous ideals is Golod  (for example, see \cite[Problem 6.18]{MP}) but a counter-example, even for monomial ideals, was constructed by the second author in \cite{D}. 

In this work we provide a concrete characterization of Golod monomial ideals in three variables, and use it to show that the product of any two proper monomial ideals in $Q=k[x,y,z]$ is Golod. The following are our main results:

\begin{theorem}\label{mainT}
Let $Q=k[x,y,z]$ and $I\subseteq \m^2$ be a monomial ideal. Then $I$ is Golod if and only if the following conditions hold:
\begin{enumerate}
\item $[I:x_1]\cdot [I:(x_2,x_3)] \subseteq I$ for all permutations $\{x_1,x_2,x_3\}$ of $\{x,y,z\}$.
\item $[I:x_1] \cdot [I:x_2]\subseteq x_3[I:(x_1,x_2)]+I$ for all permutations $\{x_1,x_2,x_3\}$ of $\{x,y,z\}$.
\end{enumerate}
\end{theorem}

To obtain Theorem \ref{mainT}, we first write down a set of necessary conditions for Golodness for general ideals in all dimensions that are easy to check and are probably of independent interest in Proposition \ref{nec}. They can be used to provide quick examples of non-Golod ideals. 

As a consequence, we obtain that products of monomial ideals in three variables are Golod.
\begin{corollary}\label{product}
Let $J,K$ be proper monomial ideals in $Q=k[x,y,z]$. Then $I=JK$ is Golod. 
\end{corollary}
Observe that this result is optimal, as the example of a non-Golod product of two monomial ideals constructed in \cite{D} is in four variables. We end the paper with some positive and negative partial results regarding the colon conditions highlighted by this work, and several open questions.

\begin{acknowledgement}
We thank Van Nguyen and Oana Veliche for many inspiring conversations about the content of this paper. The first author is partially supported by  NSA grant H98230-16-1-001 during the preparation of this work. 
\end{acknowledgement}

\section{Characterization of monomial Golod ideals in three variables}
\label{main}

In this section we prove Theorem \ref{mainT}. We first focus on the necessary part, which holds quite generally. Let $Q=k[x_1,\ldots,x_n]$, $I$ be a homogeneous ideal in $Q$, and $R=Q/I$. Let $K^Q$ be the Koszul complex on a minimal set of generators $x_1,\dots,x_n$ of the maximal ideal $\m$ of $Q$, and $K^R=R\otimes_QK^Q$. The Koszul complex can be realized as the exterior algebra $\bigwedge K_1^R$, where $K_1^R$ is a free $R$-module of rank $n$, with basis $e_{x_1},\ldots,e_{x_n}$. An element of the $p$-th graded component $K_p^R$ can be written as a sum of elements of the form $r e_{x_{i_1}\ldots x_{i_p}}$, where $1 \leq i_1 < i_1 < \ldots < i_p \leq n$, $r \in R$, and where we set $e_{x_{i_1}\ldots x_{i_p}}:=e_{x_{i_1}} \wedge \ldots \wedge e_{x_{i_p}}$. The Koszul complex also comes equipped with a differential $\partial$, as it is a DG algebra. The differential is such that $\partial(e_{x_{i_1}\ldots x_{i_p}}) = \sum_{j=1}^p (-1)^{j-1} x_{i_j} e_{x_{i_1}\ldots \widehat{x_{i_j}} \ldots x_{i_p}}$, and extended by linearity to $K_p^R$.
It is well-known that, if $I$ is Golod, then the product on the Koszul homology $H_{\geq 1}(K^R)$ is trivial (for example, see \cite[Remark 5.2.1]{Av}). In other words, $R$ is Golod only if the map $H_j(K^R)\times H_{i-j}(K^R) \to H_i(K^R)$ is zero for all $1 \leq j\leq i\leq n$. 
 
 \begin{proposition}\label{nec}
Let $Q=k[x_1,\ldots,x_n]$, and $I\subseteq \m^2$ be a homogeneous ideal such that $R=Q/I$ is Golod. Then the following hold:
 \begin{enumerate}
 \item For any $1\leq p\leq n$, we have $$[I:(x_1,\dots, x_p)][I:(x_{p+1},\dots, x_n)] \subseteq I.$$
 \item  For any $1\leq p\leq n-1$, we have $$[I:(x_1,\dots, x_p)][I:(x_{p+1},\dots, x_{n-1})] \subseteq x_n[I:(x_1,\dots,x_{n-1})]+I.$$
 \end{enumerate}
 \end{proposition}

\begin{proof}
For (1), let $f\in I:(x_1,\dots, x_p)$ and $g \in I:(x_{p+1},\dots, x_n)$. Then, by definition, the element $f e_{x_1 \ldots x_p}$ is a cycle in $K^R_p$. Similarly for $g e_{x_{p+1} \ldots x_n}$. The product of these cycles is $0$ in $H_n(R)$  if and only if $fg=0$ in $R$, which precisely says that $fg \in I$.

Similarly, take $f\in I:(x_1,\dots, x_p)$ and $g\in I:(x_{p+1},\dots, x_{n-1})$. Consider the cycles $f e_{x_1 \ldots x_p}$ and $g e_{x_{p+1}\ldots x_{n-1}}$.  The product of these is zero in $H_{n-1}(R)$ if and only if there is $h \in R$ such that $\partial(h e_{x_1\ldots x_n}) = fg e_{x_1 \ldots x_{n-1}}$. But this means that $hx_i=0$ for $1\leq i<n$ and $hx_n=fg$ in $R$. Lifting to $Q$, this shows that $h \in I:(x_1,\ldots,x_{n-1})$, and thus $fg \in x_n[I:(x_1,\ldots,x_{n-1})]+I$.
\end{proof}
 
\begin{remark}
The above proposition is motivated by the examples in \cite{D}. It can be used to easily provide examples of non-Golod ideals. For example, let $I=(x^2,y^2,z^2,t^2)(x,y,z,t)\subseteq Q=k[x,y,z,t]$. Then $xy\in I:(x,y)$ and $zt\in I:(z,t)$ but $xyzt\notin I$. Thus $I$ is not Golod. 
\end{remark}

It is well-known that, for homogeneous ideals inside polynomial rings in three variables, being Golod is equivalent to requiring that the product on the Koszul homology is trivial; for instance, see \cite[Theorem 6.3]{Ka}. In the same article, it is shown that this is not the case more generally, even for monomial ideals. In order to prove the converse of Proposition \ref{nec} for monomial ideals in $k[x,y,z]$, we show that the Koszul homology modules admit ``monomial bases". This is what we shall focus on for the rest of this section.  

\begin{definition} Let $Q=k[x_1,\ldots,x_n]$, and $I \subseteq \m^2$ be a monomial ideal. Let $R=Q/I$. We say that $H_p(K^R)$ admits a monomial basis if it has a $k$-basis consisting of classes of cycles of the form $\overline{u} e_{x_{i_1}\ldots x_{i_p}}$, where $u \in Q$ is a monomial and $\overline{u}$ denotes its image inside $R$.
\end{definition}

Clearly, if $\overline{u} e_{x_{i_1}\ldots x_{i_p}}$ is part of a monomial basis of $H_p(K^R)$, then $u \in I:(x_{i_1},\ldots,x_{i_p})$.

Observe that, if the ideal $I$ is homogeneous, then we can talk about homogeneous elements in $K^R$: if $r \in R$ is homogeneous of degree $d$, then $\deg(r e_{x_{i_1}\ldots e_{x_{i_p}}}) = d+p$. In this case, the differential preserves degrees. Even more specifically, if $I$ is monomial, then each $K_p^R$ is a $\ZZ^n$-graded $R$-module. If $r = x_1^{a_1}\cdots x_n^{a_n}$, then $r e_{x_{i_1}\ldots x_{i_p}}$ has degree $(a_1,\ldots,a_n) + \epsilon_{i_1} + \cdots + \epsilon_{i_p}$, where $\epsilon_j$ is the vector in $\ZZ^n$ which has $1$ in position $j$ and $0$ elsewhere. For example, $x^2y^3 e_{yz} \in K_2^{k[x,y,z]}$ has degree $(2,4,1)$. In this case, the differential $\partial$ on $K^R$ preserves multidegrees.

The following is well-known. Nonetheless, we provide a short proof for completeness.

\begin{proposition} \label{prop_grading}
Let $Q=k[x_1,\ldots,x_n]$ and $I \subseteq \m^2$ be a monomial ideal. Let $R=Q/I$. The modules $H_\bullet(K^R)$ admit a $\ZZ^n$-graded $k$-basis.
\end{proposition}
\begin{proof} Since $I$ is a monomial ideal, $R$ admits a graded free resolution with $\ZZ^n$-graded shifts (for example, the Taylor resolution). There is a $\ZZ^n$-graded isomorphism
\[
H_p(K^R) \cong \Tor{p}{Q}{Q/I}{k}
\]
that comes from tracing Koszul cycles along the double complex $P_\bullet \otimes K^Q$, where $P_\bullet \to R$ is a $\ZZ^n$-graded free resolution of $R$, and $K^Q$ can be viewed as a $\ZZ^n$-graded minimal free resolution of $Q/\m \cong k$. Since $Q/I$ has $\ZZ^n$-graded shifts, we see that $\Tor{p}{Q}{Q/I}{k}$ has a $\ZZ^n$-graded $k$-basis. Via this isomorphism, such a basis maps to a set of graded Koszul cycles in $K_p^R$, which forms a $k$-basis in homology.
\end{proof}

We observe that if $\sum_{\{i_1,\ldots,i_p\} \subseteq [n]} r_{i_1\ldots i_p} e_{x_{i_1}\ldots x_{i_p}} \in K_p^R$ is $\ZZ^n$-graded, then  each $r_{i_1\ldots i_p}$ must necessarily be a monomial. Furthermore, we must have $r_{i_1\ldots i_p} x_{i_1} \cdots x_{i_p} = r_{i_1'\ldots i_p'} x_{i_1'}\ldots x_{i_p'}$ for all $\{i_1,\ldots , i_p\},\{i_1',\ldots,i_p'\}$ for which $r_{i_1\ldots i_p} \ne 0$ and $r_{i_1'\ldots i_p'} \ne 0$. For example, $xe_{yz} + ye_{xz} \in K_2^{k[x,y,z]}$ is $\ZZ^3$-graded, of degree $(1,1,1)$.

\begin{lemma} \label{monomial_extremal_cycles}
Let $Q=k[x_1,\ldots,x_n]$, and $I \subseteq \m^2$ be a non-zero monomial ideal. Let $R=Q/I$. There exists a $k$-basis of $H_1(K^R)$  consisting of elements of the form $\{u e_{x_i}\}$, where $u \in I:x_i$ is a monomial. Moreover, if $\depth(R)=0$, there exists a $k$-basis of $H_n(K^R)$  consisting of elements of the form $\{\overline{u} e_{x_1\ldots x_n}\}$, where $u \in I:\m$ is a monomial in $Q$ and $\overline{u}$ denotes its residue class in $R$.
\end{lemma}

\begin{proof}
It is clear that a $k$-basis of $H_n(K^R)$ can be chosen to be of such form. In fact, an element of $K_n^R$ is of the form $\overline{f} e_{x_1\ldots x_n}$, where $f \in I:\m$. Since $I$ is monomial, we can choose $f$ to be a monomial.

It is also fairly easy to prove the claim for $H_1(K^R)$ directly. However, we explain the process via lifting Koszul cycles, as this technique will be used later. Let $(P_\bullet, \delta)$ be a graded free resolution of $R$ as a $Q$-module. As noted in Proposition \ref{prop_grading}, we have a graded isomorphism between $\Tor{1}{Q}{Q/I}k$ and $H_1(K^R)$. Let $I = (m_1,\ldots,m_t)$ be a minimal monomial generating set for $I$, and say that $m_j=x_1^{a_{1j}}\cdots x_n^{a_{nj}} =: \underline{x}^{\underline{a}_j}$. Then $\Tor{1}{Q}{Q/I}k \cong \oplus_j k(-\underline{a}_j)$. The way the isomorphism goes is as follows: 
\[
\xymatrix{
& \oplus_j Q(-\underline{a}_j) \otimes Q \ar[d]^-{\delta_1\otimes 1}\ar[r]^-{1 \otimes \partial_0} & \oplus_j Q(-\underline{a}_j) \otimes k\cong \Tor{1}{Q}{Q/I}k\\
Q \otimes K_1^Q \ar[r]^-{1 \otimes \partial_1} \ar[d]^-{\delta_0 \otimes 1} & Q \otimes Q\\
R \otimes K_1^Q \ar@{-->}[d] \\
H_1(K^R)
}
\]
More explicitly, if we take a $\ZZ^n$-graded basis element of $k(-\underline{a}_j) \subseteq \Tor{1}{Q}{Q/I}k$ and lift it to a basis element of $Q(-\underline{a}_j) \otimes Q$, then this will map down to $m_j \otimes 1\in Q \otimes Q$ under $\delta_1 \otimes 1$. If $x_j$ is any variable that divides $m_j$, say $m_j = x_jm_j'$, then $m_j \otimes 1= (1\otimes \partial)(m_j' \otimes e_{x_j})$, where $m_j' \otimes e_{x_j} \in Q\otimes K_1^Q $. Applying $\delta_0 \otimes 1$ we get the element $\overline{m_j'}e_{x_j} \in K_1^R \otimes R \otimes Q$ which is a Koszul cycle. The process ends by considering its residue class in $H_1(K^R)$, which is a $k$-basis element of the desired form. Namely, a $k$-basis element of the form $\overline{m_j'} e_{x_j}$ where $m_j' \in I:x_j$ is a monomial.
\end{proof}
\begin{proposition} \label{n=3_monomial_cycles}
Let $I \subseteq Q=k[x,y,z]$ be a non-zero monomial ideal, and let $R=Q/I$. Then, for all $1 \leq p \leq 3-\depth(R)$, the module $H_p(K^R)$ admits a $k$-basis of the form $\{\overline{u} e_{x_{i_1 \ldots x_{i_p}}}\}$, where $u \in I:(x_{i_1},\ldots,x_{i_p})$ is a monomial, and $\overline{u}$ denotes the residue class of $u$ in $R$.
\end{proposition}
\begin{proof}
The statement for $H_1(K^R)$ and $H_3(K^R)$ has already been proved in Lemma \ref{monomial_extremal_cycles} (assuming that $\depth(R)=0$ for the latter to be non-zero). The argument for $H_2(K^R)$ exploits again the process of lifting Koszul cycles. Assume that $H_2(K^R) \ne 0$, that is, $\depth(R) \leq 1$. We consider a minimal $\ZZ^3$-graded free resolution of $R$ over $Q$:
\[
\xymatrix{
F_\bullet: \ldots \ar[r] & \oplus_\ell Q(-\underline{b}_\ell) \ar[r]^-{\delta_2} & \oplus_j Q(-\underline{a}_j) \ar[r]^-{\delta_1} & Q \ar[r]^-{\delta_0} & R \ar[r] & 0.
}
\]
After fixing bases, $\delta_1$ can be represented as the matrix $[m_1,\ldots,m_t]$, where $\{m_1,\ldots,m_t\}$ is a minimal monomial generating set of $I$. On the other hand, $\delta_2$ is represented by a matrix where every column has precisely two non-zero monomial entries. This is because every relation between distinct monomials $m_i$ and $m_j$ is of this form $m_iu_i - m_ju_j=0$ for some monomials $u_i,u_j \in \m$. We now describe the lifting process for $H_2(K^R)$. Consider the following part of double complex $F_\bullet \otimes K_\bullet^Q$:
\[
\xymatrix{
&& \oplus_\ell Q(-\underline{b}_\ell) \otimes Q \ar[d]^-{\delta_2\otimes 1}\ar[r]^-{1 \otimes \partial_0} & \oplus_\ell Q(-\underline{b}_\ell) \otimes k\cong \Tor 2R{Q/I}k \\
&\oplus_j Q(-\underline{a}_j) \otimes K_1^Q \ar[r]^-{1 \otimes \partial_1} \ar[d]^-{\delta_1 \otimes 1} & \oplus_j Q(-\underline{a}_j) \otimes Q \\
Q \otimes K_2^Q \ar[d]^-{\delta_0 \otimes 1} \ar[r]^-{1 \otimes \partial_2} & Q \otimes K_1^Q \\
R \otimes K_2^Q \ar@{-->}[d]  \\
H_2(K^R)
}
\]
If we lift a $\ZZ^3$-graded basis element of $k(-\underline{b}_\ell) \subseteq \Tor 2R{Q/I}k$ to $\oplus_\ell Q(-\underline{b}_\ell) \otimes Q$, this will map down via $\delta_2 \otimes 1$ to an element $(0,\ldots,u_i,\ldots,-u_j,\ldots,0) \in \oplus_j Q(-\underline{a}_j) \cong \bigoplus_j Q(-\underline{a}_j) \otimes Q$, corresponding to a binomial relation in the $\ell$-th column of the matrix representing $\delta_2$, as described above. Write $u_i=x_iv_i$ and $u_j=x_jv_j$, for some $i,j$, and monomials $v_i,v_j \in Q$. Observe that $i \ne j$, since otherwise the relation between $m_i$ and $m_j$ given by $(0,\ldots,u_i,\ldots,-u_j,\ldots,0)$ would not be minimal. We may assume that $i<j$. Using the above relations, we have that $(0,\ldots,u_i,\ldots,-u_j,\ldots,0) = (1 \otimes \partial_1)(0,\ldots, v_ie_{x_i},\ldots,-v_je_{x_j},\ldots,0)$. We now push this element down via $\delta_1 \otimes 1$, to get an element $\sigma = v_i m_i \otimes e_{x_i} - v_j m_j \otimes e_{x_j} \in  Q\otimes  K_1^Q$.  From $x_iv_im_i = x_jv_jm_j$, we deduce that $v_im_i = -x_jw$ for some monomial $w$. Consider the element $w \otimes e_{x_ix_j} \in Q \otimes K_2^Q$; we claim that $(1 \otimes \partial_2)(w \otimes e_{x_ix_j}) = \sigma$. By definition of the differential, we have $(1 \otimes \partial_2)(w \otimes e_{x_ix_j}) = wx_i e_{x_j} - wx_j e_{x_i}$. Since $-wx_jx_i = v_im_ix_i = v_jm_jx_j$, we deduce that $wx_i = -v_jm_j$. Putting these facts together, gives $(1 \otimes \partial_2)(w \otimes e_{x_ix_j}) = v_im_i \otimes e_{x_i} - v_jm_j \otimes e_{x_j} = \sigma$, as desired. As $(\delta_0 \otimes 1)(w \otimes e_{x_ix_j} = \overline{w} \otimes e_{x_ix_j}$ is a cycle in $R \otimes K_2^Q$, the process of lifting Koszul cycles now ends by considering the class of the element $\overline{w} e_{x_ix_j}$ inside $H_2(K^R)$. As observed above, inside $Q$ we have $wx_i = -v_jmj \in I$, and $wx_j = -v_im_i \in I$. Therefore $w \in I:(x_i,x_j)$, and the class of $\overline{w} e_{x_ix_j}$ inside $H_2(K^R)$ gives then a basis element of the desired form.
\end{proof}

\begin{proof}[Proof of Theorem \ref{mainT}]
The necessary part was Proposition \ref{nec}. But as Proposition \ref{n=3_monomial_cycles} shows that all Koszul homologies admit a $k$-basis of the form $\{u e_{x_{i_1 \ldots x_{i_p}}}\}$, where $u \in I:(x_{i_1},\ldots,x_{i_p})$ is a monomial and $p=1,2,3$, the stated conditions are also sufficient. 
\end{proof}

We observe that the condition that $I \subseteq \m^2$ in Theorem \ref{mainT} cannot be removed.\begin{example}
Consider the ideal $I=(x,y^2,yz,z^2)$ inside $Q = k[x,y,z]$. This ideal does not satisfy the second condition of Theorem \ref{mainT}, since $y \in [I:x] \cdot [I:y] = I:y$, but $y \notin z[I:(x,y)]+I = I$. However, the ring $Q/I \cong k[y,z]/(y,z)^2$ is Golod.
\end{example}

\begin{remark}
Let $I$ be a monomial ideal in $k[x_1,\dots, x_n]$. The condition ``strongly Golod" considered by Herzog and Huneke in \cite{HH} means, in this context, that $[I:(x_i)][I:(x_j)]\subseteq I$ for all $1\leq i,j\leq n$. This condition is clearly stronger than all the necessary colon conditions in \ref{nec}. This makes sense, since strongly Golod implies Golod. 
\end{remark}

We conclude this section observing that, in the case of monomial ideals in four or more variables, some Koszul homology modules may not always admit a ``monomial basis''.
\begin{example}
Consider the ideal $I=(xz,xw,yz,yw,x^2,y^2,z^2,w^2)$ in $Q=k[x,y,z,w]$, and let $R=Q/I$. It is easy to check that $\alpha = xe_{yzw}-ye_{xzw}$ is a cycle of $K^R_3$, whose class equals that of $we_{xyz}-ze_{xyw}$ in homology. Observe that $\alpha$ has multidegree $(1,1,1,1)$. Exploiting the multigrading, one can show that the class of $\alpha$ in homology cannot be expressed as a combination of elements coming from a monomial basis of $H_3(K^R)$.
\end{example}

\section{Products of monomial ideals in $k[x,y,z]$ are Golod}

In this section we prove Corollary \ref{product}. By Theorem \ref{mainT} and symmetry, it suffices to show the following lemmas. 

\begin{lemma}\label{jk1}
Let $J,K$ be proper monomial ideals in $Q=k[x,y,z]$ and $I=JK$. Then $$[I:x][I:(y,z)]\subseteq I.$$
\end{lemma}

\begin{proof}
Let $f \in I:x$ and $g\in I:(y,z)$ be monomials. If $f \in (y,z)$ then $fg\in I$, so we assume $f=x^a$ for some $a\geq 0$. It follows that $x^{a+1}\in JK$, and as $J,K$ are proper we must have $f\in J \cap K$. As $J,K$ are monomial ideals, we have $g\in JK:y \subseteq (J:y)K+(K:y)J \subseteq J+K$, and we are done. 

\end{proof}

\begin{lemma}\label{jk2}
Let $J,K$ be proper monomial ideals in $Q=k[x,y,z]$ and $I=JK$. Then $$(I:x)(I:y)\subseteq z[I:(x,y)]+I.$$
\end{lemma}

\begin{proof}
As $J$ and $K$ are monomial ideals, we have $JK:x= (J:x)K+(K:x)J$. We then have
\[
(JK:x)(JK:y) = J^2(K:x)(K:y) + K^2(J:x)(J:y) +JK(J:x)(K:y) +JK(K:x)(J:y).
\]
By symmetry, it is enough to show that  $J^2(K:x)(K:y) \subseteq z[JK:(x,y)] +JK$.

Let $J=J_1+zJ'$ with $J_1$ generated by $J\cap k[x,y]$. Then 
\[
J^2(K:x)(K:y) \subseteq J^2[K:(x,y)]= [J_1 +zJ']J[K:(x,y)].
\]
But 
\[
J_1J[K:(x,y)]\subseteq (x,y)J[K:(x,y)] \subseteq JK
\]
and   
\[
zJ'J[K:(x,y)] \subseteq zJ[K:(x,y)]\subseteq z[JK:(x,y)].
\]
\end{proof}

\section{Integrally closed ideals and some questions}

The colon conditions considered in this paper seem related to the property of ``being integrally closed" (see also the $\m$-full and basically full conditions \cite{Wa, HRR}). Here we give some positive and negative results in this direction. 

\begin{lemma}\label{inte}
Let $J,K$ be homogenous ideals in $Q=k[x,y,z]$ ($J$ may not be proper) and $\m=(x,y,z)$. Assume that $K\subseteq \m^2$ and $K$ is integrally closed. Then if $I=JK$, we have $[I:x][I:(y,z)]\subseteq I$.
\end{lemma}

\begin{proof}
Write $(J,y,z) = (x^a,y,z)$ and $(K,y,z)=(x^b,y,z)$ for some $a,b$, with $b \geq 2$ by assumption. Let $f \in JK:x$, so that $fx \in JK \subseteq (J,y,z)(K,y,z) \subseteq (x^{a+b},y,z)$. It follows that $f \in (x^{a+b-1},y,z)$, write $f=ux^{a+b-1}+v$ for some $v \in (y,z)$. Now let $g \in JK:(y,z)$, and consider the element $h=ux^{b-1}g$. As we are assuming $b \geq 2$, we have $h \in \m JK:(y,z)$. On the other hand, since we have $fx = ux^{a+b}+xv \in JK$,  we get that $x^{a+1}h = ux^{a+b}g = fxg - xvg \in \m JK$. It follows that $h \in \m JK:(x^{a+1},y,z)$. Observe that $(\m J,y,z) = (x^{a+1},y,z)$. It follows that $h \in \m JK:(x^{a+1},y,z) \subseteq \m JK:\m J = K$, because $K$ is integrally closed. Since $f$ is congruent to $ux^{a+b-1}$ modulo $(y,z)$, we have that $fg$ is congruent to $ux^{a+b-1}g = x^ah$ modulo $JK$. On the other hand, $x^a$ belongs to $(J,y,z)$, therefore $x^ah \in (J,y,z) [K \cap JK:(y,z)] \subseteq JK$.
 \end{proof}
 
 \begin{corollary}
 Let $I$ be a homogenous ideal in $Q=k[x,y,z]$ and $\m=(x,y,z)$. Assume that $I\subseteq \m^2$ and $I$ is integrally closed. Then  we have $[I:x][I:(y,z)]\subseteq I$.
\end{corollary}

\begin{proof}
Let $J=R$ in Lemma \ref{inte}.
\end{proof}

 
Unfortunately, one can use other necessary colon criteria provided in \ref{nec} to show that even integrally closed {\it monomial} ideal in three variables or product of them in four variables may not be Golod. 

\begin{example} \label{ex1}
Let $I = \overline{(x^2, y^4, z^4, yz)} = (x^2,y^4,z^4, xz^2, yz, xy^2)$ in $Q=k[x,y,z]$. Then $xz \in (I:x)(I:y)$ but it is not in $z[I:(x,y)]+I$. So $I$ cannot be Golod by \ref{nec}. If one does not want to restrict to $\m$-primary ideals, a simpler example of an integrally closed ideal in $k[x,y,z]$ that is not Golod is $I=(x^2,yz)$. Indeed, this ideal fails again the condition $(I:x)(I:y) \subseteq z[I:(x,y)]+I$ of Lemma \ref{nec}; moreover, it is a complete intersection of height two.
\end{example}
\begin{example} \label{ex2} Let $J=  \overline{(x^2, y^4, z^2, yz)}$ and $K = \overline{(x^4, y^2, w^2, xw)}$ in $Q=k[x,y,z,w]$. Using Macaulay 2 \cite{M2}, one can check that $JK$ is not Golod.
\end{example}
Note that both Examples \ref{ex1} and \ref{ex2} are in the smallest possible number of variables.

To end this paper, we pose some intriguing question motivated by our work. The obvious one is:

\begin{question}\label{ques1}
Let $J,K$ be proper homogeneous ideals in $Q=k[x,y,z]$. Is $JK$ Golod? 
\end{question}

We do not know the answer to Question \ref{ques1} even when $K=\m$. One can show that the conclusions of Lemma \ref{jk1} and \ref{jk2} still hold when $K=\m$ and $J$ is any proper homogeneous ideal in $Q$, so Proposition \ref{nec} does not provide any obstructions in this case. When the characteristic of $k$ is $0$ and $J=K$, the answer is positive by the main result of \cite{HH}.

Finally, we have not been able to determine whether  Lemma \ref{jk1} holds for any product $JK$ of homogeneous ideals in three variables, without either the monomial condition, or assuming that $K \subseteq \m^2$ is integrally closed. Observe that any example for which the lemma fails would provide a negative answer to Question \ref{ques1}. It is rather frustrating that such a simple-looking question cannot be resolved, so the first author is willing to offer a cash prize of $25$ USD for the first solver of this:

\begin{question}
Let $J,K$ be proper homogeneous ideals in $Q=k[x,y,z]$. Is this true that $$[I:x][I:(y,z)]\subseteq I$$ for $I=JK$?
\end{question}








\bibliographystyle{alpha}
\bibliography{References}

\begin{thebibliography}{GPTW16}

\bibitem[Avr98]{Av}
Luchezar~L. Avramov.
\newblock Infinite free resolutions.
\newblock In {\em Six lectures on commutative algebra ({B}ellaterra, 1996)},
  volume 166 of {\em Progr. Math.}, pages 1--118. Birkh\"{a}user, Basel, 1998.

\bibitem[CV18]{CV}
Lars~Winther Christensen and Oana Veliche.
\newblock The {G}olod property of powers of the maximal ideal of a local ring.
\newblock {\em Arch. Math. (Basel)}, 110(6):549--562, 2018.

\bibitem[DS07]{DS}
Graham Denham and Alexander~I. Suciu.
\newblock Moment-angle complexes, monomial ideals and {M}assey products.
\newblock {\em Pure Appl. Math. Q.}, 3(1, Special Issue: In honor of Robert D.
  MacPherson. Part 3):25--60, 2007.

\bibitem[DS16]{D}
Alessandro De~Stefani.
\newblock Products of ideals may not be {G}olod.
\newblock {\em J. Pure Appl. Algebra}, 220(6):2289--2306, 2016.

\bibitem[Fra18]{Fr}
Robin Frankhuizen.
\newblock {$A_\infty$}-resolutions and the {G}olod property for monomial rings.
\newblock {\em Algebr. Geom. Topol.}, 18(6):3403--3424, 2018.

\bibitem[GPTW16]{GPTW}
Jelena Grbi\'{c}, Taras Panov, Stephen Theriault, and Jie Wu.
\newblock The homotopy types of moment-angle complexes for flag complexes.
\newblock {\em Trans. Amer. Math. Soc.}, 368(9):6663--6682, 2016.

\bibitem[GS]{M2}
Daniel~R. Grayson and Michael~E. Stillman.
\newblock Macaulay2, a software system for research in algebraic geometry.
\newblock Available at \url{http://www.math.uiuc.edu/Macaulay2/}.

\bibitem[HH13]{HH}
J\"{u}rgen Herzog and Craig Huneke.
\newblock Ordinary and symbolic powers are {G}olod.
\newblock {\em Adv. Math.}, 246:89--99, 2013.

\bibitem[HRR02]{HRR}
William~J. Heinzer, Louis~J. Ratliff, Jr., and David~E. Rush.
\newblock Basically full ideals in local rings.
\newblock {\em J. Algebra}, 250(1):371--396, 2002.

\bibitem[IK18]{IK}
Kouyemon Iriye and Daisuke Kishimoto.
\newblock Golodness and polyhedral products of simplicial complexes with
  minimal {T}aylor resolutions.
\newblock {\em Homology Homotopy Appl.}, 20(1):69--78, 2018.

\bibitem[Kat17]{Ka}
Lukas Katth\"{a}n.
\newblock A non-{G}olod ring with a trivial product on its {K}oszul homology.
\newblock {\em J. Algebra}, 479:244--262, 2017.

\bibitem[MP15]{MP}
Jason McCullough and Irena Peeva.
\newblock Infinite graded free resolutions.
\newblock In {\em Commutative algebra and noncommutative algebraic geometry.
  {V}ol. {I}}, volume~67 of {\em Math. Sci. Res. Inst. Publ.}, pages 215--257.
  Cambridge Univ. Press, New York, 2015.

\bibitem[Wat87]{Wa}
Junzo Watanabe.
\newblock {$\mathfrak{m}$}-full ideals.
\newblock {\em Nagoya Math. J.}, 106:101--111, 1987.

\end{thebibliography}
\end{document}